\documentclass{amsart}
\usepackage{amssymb}

\newcommand{\IN}{\mathbb N}
\newcommand{\IR}{\mathbb R}
\newcommand{\IQ}{\mathbb Q}

\newcommand{\id}{\mathrm{id}}

\newcommand{\w}{\omega}
\newcommand{\C}{\mathcal C}

\newcommand{\Ra}{\Rightarrow}

\newcommand{\cl}{\mathrm{cl}}

\newcommand{\concat}{\hat{\phantom{v}}}

\newtheorem{theorem}{Theorem}[section]
\newtheorem{lemma}[theorem]{Lemma}
\newtheorem{proposition}[theorem]{Proposition}

\newtheorem{problem}[theorem]{Problem}
\theoremstyle{definition}

\newtheorem{ex}[theorem]{Example}

\begin{document}

\title[Weakly discontinuous and resolvable functions]{Weakly discontinuous and resolvable functions\\ between topological spaces}
\author{Taras Banakh}
\author{Bogdan Bokalo}
\address{Instytut Matematyki, Jan Kochanowski University in Kielce (Poland) and\newline
Department of Mathematics, Ivan Franko National University of Lviv (Ukraine)}
\email{t.o.banakh@gmail.com, b.m.bokalo@gmail.com}
\dedicatory{Dedicated to the memory of Prof. Dr.~L.~Michael Brown}
\begin{abstract} We prove that a function $f:X\to Y$ from a first-countable (more generally, Preiss-Simon) space $X$ to a regular space $Y$ is weakly discontinuous (which means that every subspace $A\subset X$ contains an open dense subset $U\subset A$ such that $f|U$ is continuous) if and only if $f$ is open-resolvable (in the sense that for every open subset $U\subset Y$ the preimage $f^{-1}(U)$ is a resolvable subset of $X$) if and only if $f$ is resolvable (in the sense that for every resolvable subset $R\subset Y$ the preimage $f^{-1}(R)$ is a resolvable subset of $X$). For functions on metrizable spaces this characterization was announced (without proof) by Vinokurov in 1985.
\end{abstract}

%\begin{document}
\maketitle
%\tableofcontents

\section{Introduction and Main Result}\label{s7}

In this paper we present a proof of a characterization of weakly discontinuous functions announced (without proof) by Vinokurov in \cite{Vino}.

A function $f:X\to Y$ between topological spaces is called {\em weakly discontinuous} if every subspace $A\subset X$ contains a dense open subset $U\subset A$ such that the restriction $f|U$ is continuous. It is well-known that for weakly discontinuous maps $f:X\to Y$ and $g:Y\to Z$ the composition $g\circ f:X\to Z$ is weakly discontinuous. Weakly discontinuous functions were introduced by Vinokurov \cite{Vino}. Many properties of functions, equivalent to the weak discontinuity were discovered in \cite{AB,BKMM,BM,CL1,CL2,JR82,KM,Kir,O'M,Sol,Vino}. By \cite[Theorem 8]{Vino}, a function $f:X\to Y$ from a metrizable space $X$ to a regular space $Y$ is weakly discontinuous if and only if for every open set $U\subset Y$ the preimage $f^{-1}(U)$ is a resolvable subset of $X$. We recall \cite[I.12]{Kur} that a subset $A$ of a topological space $X$ is {\em resolvable} if for every closed subset $F\subset X$ the set $\overline{F\cap A}\cap \overline{F\setminus A}$  is nowhere dense in $F$. Observe that a subset $A\subset X$ is resolvable if and only if its characteristic function $\chi_A:X\to\{0,1\}$ is weakly discontinuous.
It is known \cite[I.12]{Kur} that the family resolvable subsets of a topological space $X$ is closed under intersections, unions, and complements.

A function $f:X\to Y$ between topological spaces is called ({\em open-}){\em resolvable} if for every (open) resolvable subset $R\subset Y$ the preimage $f^{-1}(R)$ is a resolvable subset of $X$. It is clear that each resolvable function is open-resolvable.

\begin{proposition}\label{p1} If a function $f:X\to Y$ between topological spaces is weakly discontinuous, then $f$ is resolvable.
\end{proposition}

\begin{proof} If a subset $A\subset Y$ is resolvable, then its characteristic function $\chi_A:Y\to\{0,1\}$ is weakly discontinuous. Since the weak discontinuity is preserved by compositions (see, e.g., \cite[4.1]{BB}), the composition $g=\chi_A\circ f:X\to\{0,1\}$ is weakly discontinuous, which implies that the set $g^{-1}(1)=f^{-1}(A)$ is resolvable in $X$.
\end{proof}

By \cite[Theorem 8]{Vino}, for functions between metrizable spaces, Proposition~\ref{p1} can be reversed. However the paper \cite{Vino} contains no proof of this important fact. In this paper we present a proof of this Vinokurov's characterization in a more general case of functions defined on Preiss-Simon spaces.

We define a topological space $X$ to be {\em Preiss-Simon} at a point $x\in X$ if for any subset $A\subset X$ with $x\in \overline{A}$ there is a sequence $(U_n)_{n\in\w}$ of non-empty open subsets of $A$ that converges to $x$ in the sense that each neighborhood of $x$ contains all but finitely many sets $U_n$. By $PS(X)$ we denote the set of points $x\in X$ at which $X$ is Preiss-Simon. A topological space $X$ is called a {\em Preiss-Simon} space if $PS(X)=X$ (that is $X$ is Preiss-Simon at each point $x\in X$).

It is clear that each first-countable space is Preiss-Simon and each Preiss-Simon space is Fr\'echet-Urysohn. A less trivial fact due to Preiss and Simon \cite{PS} asserts that each Eberlein compact space is Preiss-Simon.

A base $\mathcal B$ of the topology of a space $X$ will be called {\em countably additive} if the union $\cup\C$ of any countable subfamily $\C\subset\mathcal B$ belongs to $\mathcal B$.

A function $f:X\to Y$ between topological spaces will be called {\em base-resolvable} if there exists a countably additive base $\mathcal B$ of the topology of $Y$ such that for every set $B\subset Y$ the preimage $f^{-1}(B)$ is a resolvable subset of $X$.

It is clear that for any function $f:X\to Y$ we have the implications:
$$\mbox{weakly discontinuous $\Ra$ resolvable $\Ra$ open-resolvable $\Ra$ base-resolvable}.$$

For functions on Preiss-Simon spaces these implications can be reversed, which is proved in the following characterization. For functions on metrizable spaces it was announced (without written proof) by Vinokurov in \cite[Theorem 8]{Vino}.

\begin{theorem}\label{main} For a functions $f:X\to Y$ from a Preiss-Simon space $X$ to a regular space $Y$ the following conditions are equivalent:
\begin{enumerate}
\item $f$ is weakly discontinuous;
\item $f$ is resolvable;
\item $f$ is open-resolvable;
\item $f$ is base-resolvable.
\end{enumerate}
\end{theorem}

This theorem will be proved in Section~\ref{s:pf} after some preliminary work made in Section~\ref{lemmas}.

By Theorem~\ref{main}, any open-resolvable map $f:X\to Y$ from a Preiss-Simon space $X$ to a regular space $Y$ is resolvable. We do not know if this implication still holds for any function between regular spaces. The authors are grateful to Sergey Medvedev for turning their attention to this intriguing question.

\begin{problem}[Medvedev]\label{Med} Is each open-resolvable function $f:X\to Y$ between regular spaces resolvable?
\end{problem}

The following example indicates that Problem~\ref{Med} can be difficult and shows that the countable additivity of the base $\mathcal B$ cannot be removed from the definition of a base-resolvable function.

\begin{ex} Let $\IR_{\IQ}$ be the real line endowed with the metrizable topology generated by the countable base $\mathcal B=\big\{(a,b):a<b,\;a,b\in\IQ\big\}\cup \big\{\{q\}:q\in\IQ\big\}$. The identity map $\id:\IR\to\IR_\IQ$ is not (open-)resolvable as the preimage $\IQ=\id^{-1}(\IQ)$ of the open set $\IQ\subset\IR_\IQ$ is not resolvable in $\IR$. Yet, for every basic set $B\in\mathcal B$ the preimage $\id^{-1}(B)$ is a resolvable set in $\IR$.
\end{ex}

\section{Five Lemmas}\label{lemmas}

In this section we shall prove some auxiliary results, which will be used in the proof of Theorem~\ref{main}. For a subset $A$ of a topological space by $\bar A$, $A^\circ$, and $\overline{A}^\circ$ we denote the closure, the interior, and the interior of the closure of $A$ in $X$, respectively. A family $\mathcal B$ of non-empty open subsets of a topological space $X$ is called a {\em $\pi$-base}
if each non-empty open set $U\subset X$ contains some set $B\in\mathcal B$.

Following \cite{BB}, we define a function $f:X\to Y$ between topological spaces to be {\em scatteredly continuous} if for every non-empty subset $A\subset X$ the restriction $f|A$ has a continuity point. It is easy to see that each weakly discontinuous function is scatteredly continuous.
For maps into regular spaces the converse implication is also true (see \cite{AB}, \cite{BM} or \cite[4.4]{BB}):

\begin{lemma}\label{l:scat} A function $f:X\to Y$ from a topological space $X$ to a regular space $Y$ is weakly discontinuous if and only if $f$ is scatteredly continuous.
\end{lemma}

We recall that for a topological space $X$ its {\em tightness} $t(X)$ is the smallest cardinal $\kappa$ such that for every subset $A\subset X$ and point $a\in\bar A$ there exists a subset $B\subset A$ of cardinality $|B|\le\kappa$ such that $a\in\bar B$. The following lemma was proved in \cite[2.3]{BB}.

\begin{lemma}\label{l:count} A function $f:X\to Y$ between topological spaces is scatteredly continuous if and only if  for any non-empty subset $A\subset X$ of cardinality $|A|\le t(X)$ the restriction $f|A$ has a continuity point.
\end{lemma}

\begin{lemma}\label{resolve} If $A,B$ are disjoint resolvable subsets of a topological space $X$, then $\bar A\cap \bar B$ is nowhere dense in $X$.
\end{lemma}

\begin{proof} To derive a contradiction, assume that the set $F=\bar A\cap\bar B$ has a non-empty interior $U$ in $X$. Then $U\cap A$ and $U\cap B$ are two dense disjoint sets in $U$. By the resolvability of $A$, the dense subset $A\cap \bar U$ of $\bar U$ has nowhere dense boundary in $\bar U$. Consequently, the interior $U_A$ of the set $A\cap \bar U$ is dense in $\bar U$. By the same reason, the interior $U_B$ of the set $B\cap\bar U$ is dense in $\bar U$. Then the non-empty space $\bar U$ contains two disjoint dense open sets $U_A$ and $U_B$, which is not possible.
\end{proof}

A function $f:X\to Y$ between topological spaces is defined to be
{\em almost continuous} ({\em weakly continuous}) at a point $x\in X$ if for any neighborhood $Oy\subset Y$ of the point $y=f(x)$ the (interior of the) set $f^{-1}(Oy)$ is dense in some neighborhood of the point $x$ in $X$. By $AC(f)$ (resp. $WC(f)$~) we shall denote the set of point of almost (resp.  weak-) continuity of $f$.

\begin{lemma}\label{7.2} Let $f:X\to Y$ be a base-resolvable map from a topological space $X$ to a Hausdorff space $Y$. Then
\begin{enumerate}
\item $AC(f)=WC(f)$.
\item If $D$ is dense in $X$, $Y$ is regular, and $f|D$ has no continuity point, then $D\setminus AC(f)$ also is dense in $X$.
\item If $X$ has a countable $\pi$-base, then for any countable dense set $D\subset X$ there is a point $y\in f(D)$ such that for every neighborhood $Oy$ of $y$ the preimage $f^{-1}(Oy)$ has non-empty interior in $X$.
\item The family $\{\overline{f^{-1}(y)}^\circ:y\in Y\}$ is disjoint.
\end{enumerate}
\end{lemma}

\begin{proof} Since $f$ is basic-resolvable, there exists a countably additive base $\mathcal B$ of the topology of $Y$ such that for every $U\in\mathcal B$ the preimage $f^{-1}(U)$ is resolvable in $X$.
\smallskip

1. The inclusion $WC(f)\subset AC(f)$ is trivial. To prove that $AC(f)\subset WC(f)$, take any point $x\in AC(f)$. To show that $x\in WC(f)$, take any neighborhood $Oy\in\mathcal B$ of the point $y=f(x)$ and consider the preimage $f^{-1}(Oy)$. Since $x\in AC(f)$, the closure $F=\overline{f^{-1}(Oy)}$ is a  neighborhood of $x$. Since the set $f^{-1}(Oy)$ is resolvable, the boundary
 $\overline{F\cap f^{-1}(Oy)}\cap \overline{F\setminus f^{-1}(Oy)}$ is nowhere dense in $F$. Consequently, the interior of the set $F\cap f^{-1}(O_y)$ in $F$ is dense in $F$ and $x\in WC(f)$.
\smallskip

2. Assume that $D\subset X$ is dense, $Y$ is regular, and $f|D$ has no continuity point. Given a point $x\in D$, and a neighborhood $O_x\subset X$ of $x$ we should find a point $x'\in O_x\cap D\setminus AC(f)$. If $x\notin AC(f)$, then we can take $x'=x$. So we assume that $x\in AC(f)$ and hence $x\in WC(f)$ by the preceding item. Since $x$ is a discontinuity point of $f|D$, there is a neighborhood $O_{f(x)}$ of $f(x)$ such that $f(D\cap U_x)\not\subset O_{f(x)}$ for every neighborhood $U_x$ of $x$. Using the regularity of $Y$ choose a neighborhood $U_{f(x)}\subset Y$ of $f(x)$ with $\overline{U}_{f(x)}\subset O_{f(x)}$. Since $f$ is  weakly continuous at $x$, the closure  of the interior  of the preimage $f^{-1}(U_{f(x)})$ contains some open neighborhood $W_x$ of $x$. By the choice of $O_{f(x)}$, we can find a point $x'\in D\cap O_x\cap W_x$ with $f(x')\notin O_{f(x)}$. Consider the neighborhood $O_{f(x')}=Y\setminus \overline{U}_{f(x)}$ of $f(x')$ and observe that $W_x\cap f^{-1}(O_{f(x')})$ is a nowhere dense subset of $Ox$ (because it misses the interior of $f^{-1}(U_{f(x)})$ which is dense in $W_x$). This witnesses that $x'\notin AC(f)$.
\smallskip

3. Assume that $X$ has countable $\pi$-base $\{W_n\}_{n\in\w}$. We lose no generality assuming that the subfamilies $\{W_{2n}\}_{n\in\w}$ and $\{W_{2n+1}\}_{n\in\w}$ are countable $\pi$-bases in $X$.
Given a countable dense subset $D\subset X$, we should find a point $y\in f(D)$ such that for every neighborhood $Oy\subset Y$ the preimage $f^{-1}(Oy)$ has non-empty interior in $X$. Assume conversely that each point $y\in f(D)$ has a neighborhood $O_y\in\mathcal B$ such that the preimage $f^{-1}(O_y)$ has empty interior in $X$. The resolvability of $f^{-1}(O_y)$ implies that this set is nowhere dense in $X$.
We shall inductively construct a sequence $(x_n)_{n\in\w}$ of points of $D$ and a sequence $(U_n)_{n\in\w}$ of open sets in $Y$ such that
\begin{itemize}
\item[(a)] $f(x_n)\in U_n\in\mathcal B$ and the set $f^{-1}(U_n)$ is nowhere dense in $X$;
\item[(b)] $x_n\in D\cap W_{n}\setminus\bigcup_{k<n}f^{-1}(U_n)$;
\item[(c)] $U_n\cap\{f(x_k)\}_{k<n}=\emptyset$.
\end{itemize}

Taking into account that $\{W_{2n}\}_{n\in\w}$ and $\{W_{2n+1}\}_{n\in\w}$ are $\pi$-bases in $X$, we conclude that the disjoint sets $\{x_{2n}\}_{n\in\w}$ and $\{x_{2n+1}\}_{n\in\w}$ are dense in $X$.
The countable additivity of the base $\mathcal B$ guarantees that the open sets
$U_{e}=\bigcup_{n\in\w}U_{2n}$ and $U_o=\bigcup_{n\in\w}U_{2n+1}$ belong to $\mathcal B$. Then their preimages $f^{-1}(U_e)\supset\{x_{2n}\}_{n\in\w}$ and $f^{-1}(U_e)\supset\{x_{2n+1}\}_{n\in\w}$ are disjoint dense resolvable sets in $X$. But this contradicts Lemma~\ref{resolve}.
\smallskip

4. Assuming that the family $\{\overline{f^{-1}(y)}^\circ:y\in Y\}$ is not disjoint, find two distinct points $y,z\in Y$ such that the intersection
$$W=\overline{f^{-1}(y)}^\circ\cap \overline{f^{-1}(z)}^\circ$$is not empty.
Observe that the sets $W\cap f^{-1}(y)$ and $W\cap f^{-1}(z)$ both are dense in $W$.

By the Hausdorff property of $Y$ the points $y,z$ have disjoint open neighborhoods $Oy,Oz\in\mathcal B$.  The choice of $\mathcal B$ guarantees that the sets $f^{-1}(Oy)$ and $f^{-1}(Oz)$ are resolvable. By Lemma~\ref{resolve}, the intersection $\overline{f^{-1}(Oy)}\cap\overline{f^{-1}(Oz)}$ is nowhere dense in $X$, which is not possible as this intersection contains the non-empty open set $W$.
\end{proof}

\begin{lemma}\label{7.2a} Let $f:X\to Y$ be a base-resolvable map from a topological space $X$ to a regular  space $Y$ and $D$ be a countable dense subset of $X$ such that $f|D$ has no continuity point.
\begin{enumerate}
\item For any finite subset $F\subset Y$ there is a dense subset $Q\subset D\setminus f^{-1}(F)$ in $X$ such that $f|Q$ has no continuity point.
\item If $X$ has a countable $\pi$-base, then for any sequence $(U_n)_{n=1}^\infty$ of non-empty open subsets of $X$ there are an infinite subset $I\subset\IN$ and sequences $(V_n)_{n\in I}$ and $(W_n)_{n\in I}$ of pairwise disjoint non-empty open sets in $X$ and $Y$, respectively, such that  $V_n\subset U_n\cap f^{-1}(W_n)$ for all $n\in I$.
\item If $D\subset PS(X)$, then there is a countable first countable subspace $Q\subset D$ such that $Q$ contains no finite non-empty open subsets and the restriction $f|Q$ is a bijective map whose image $f(Q)$ is a discrete subspace of $Y$.
\end{enumerate}
\end{lemma}

\begin{proof} Fix a countably additive base $\mathcal B$ of the topology of $Y$ such that for every $U\in\mathcal B$ the preimage $f^{-1}(U)$ is resolvable in $X$.
\smallskip

1. The first statement will be proved by induction on the cardinality $|F|$ of the set $F$. If $|F|=0$, then we can put $Q=D$ and finish the proof. Assume that for some $n>0$ the first statement is proved for all sets $F\subset Y$ of cardinality $|F|<n$. Take any finite subset $F\subset Y$ of cardinality $|F|=n$. Choose any point $y\in F$. By the inductive hypothesis, for the set $F\setminus\{y\}$ there exists a dense subset $E\subset D\setminus f^{-1}(F\setminus\{y\})$ such that the function $f|E$ has no continuity points.
We claim that the set $E\setminus f^{-1}(y)$ is dense in $X$. In the opposite case, there exists a non-empty open set $U\subset X$ such that $E\cap U\subset f^{-1}(y)$. It follows that $E\cap U\subset AC(f)$. By Lemma~\ref{7.2}(2), the set $E\setminus AC(f)\subset E\setminus U$ is dense in $X$, which is a desired contradiction showing that the set $E\setminus f^{-1}(y)$ is dense in $X$ and so is the set
$Q:=(E\setminus f^{-1}(y))\setminus\big(\overline{E\cap f^{-1}(y)}\setminus\overline{E\cap f^{-1}(y)}^\circ\big)\subset D\setminus f^{-1}(F).$

It remains to check that the restriction $f|Q$ has no continuity points. To derive a contradiction, assume that some point $x_0\in Q$ is a continuity point of the restriction $f|Q$. If $x_0\notin\overline{E\cap f^{-1}(y)}$, then the discontinuity of the map $f|E$ at $x_0$ implies the discontinuity of $f|Q$ at $x_0$. So, $x_0$ belongs to the interior $\overline{E\cap f^{-1}(y)}^\circ$ of $\overline{E\cap f^{-1}(y)}$.

Let $y_0=f(x_0)$ and observe that $y_0\ne y$ (because $x_0\notin f^{-1}(y)$). By the Hausdorff property of $Y$ the points $y_0$ and $y$ have disjoint open neighborhoods $Oy_0,Oy\in\mathcal B$.
By the continuity of $f|Q$ at $x_0$, there is an open neighborhood $Ox_0\subset \overline{E\cap f^{-1}(y)}^\circ$ of $x_0$ such that $f(Ox_0\cap Q)\subset Oy_0$. It follows that the preimages
 $f^{-1}(Oy_0)$ and $f^{-1}(Oy)$ are disjoint resolvable subsets of $X$. The density of the set $Q$ in $X$  implies the density of the set
$Ox_0\cap f^{-1}(Oy_0)\supset Ox_0\cap Q$ in $Ox_0$.
On the other hand, the intersection $f^{-1}(y)\cap \overline{E\cap f^{-1}(y)}^\circ$ is dense in $\overline{E\cap f^{-1}(y)}^\circ$ and hence $Ox_0\cap f^{-1}(Oy)\supset f^{-1}(y)\cap Ox_0$ is dense in $Ox_0$. Therefore, $\overline{f^{-1}(Oy_0)}\cap\overline{f^{-1}(Oy)}$ contains the non-empty open set $Ox_0$.
But this contradicts Lemma~\ref{resolve}.
\smallskip

2. Assume that the space $X$ has a countable $\pi$-base and let $(U_n)_{n=1}^\infty$ be a sequence of non-empty open subsets of $X$.
Applying Lemma~\ref{7.2}(3) to the map $f|U_1$ and the dense subset $D\cap U_1$, find a point $y_0\in f(D\cap U_1)$ such that for each neighborhood $Oy_0$ the preimage $U_1\cap f^{-1}(Oy_0)$ has non-empty interior.
By induction, for every $n\in\IN$ we shall find a point $y_n\in f(D)\setminus \{y_i:i<n\}$ such that for every neighborhood $Oy_n$ the set $U_n\cap f^{-1}(Oy_n)$ has non-empty interior.

Assuming that for some $n$ the points $y_0,\dots,y_{n-1}$ have being chosen, we shall find a point $y_n$. It follows that the intersection $D\cap U_n$ is a countable dense subset of $U_n$ such that $f|D\cap U_n$ has no continuity point. Applying Lemma~\ref{7.2a}(1), we can find a dense subset $Q\subset D\cap U_n\setminus f^{-1}(\{y_0,\dots,y_{n-1}\})$ in $U_n$ such that the restriction $f|Q$ has no continuity point. Applying Lemma~\ref{7.2}(3) to the map $f|U_n$ and the dense subset $Q\cap U_n$ of $U_n$, find a point $y_n\in f(Q)\subset f(D)\setminus\{y_i:i<n\}$ such that for each neighborhood $Oy_n$ the preimage $U_n\cap f^{-1}(Oy_n)$ has non-empty interior.
This completes the inductive construction.
\smallskip

The space $\{y_n:n\in\IN\}$, being infinite and regular, contains an infinite discrete subspace $\{y_n:n\in I\}$. By induction, we can select pairwise disjoint open neighborhoods $W_n\subset Y$, $n\in I$, of the points $y_n$. For every $n\in I$, the choice of the point $y_n$ guarantees that the set $U_n\cap f^{-1}(W_n)$ contains a non-empty open set $V_n$. Then the set $I\subset\IN$ and sequences $(V_n)_{n\in I}$, $(W_n)_{n\in I}$ satisfy our requirements.
\smallskip

3. Assume that $D\subset PS(X)$. Using the density of the countable set $D$ and the inclusion $D\subset PS(X)$, we can show that the space $X$ has a countable $\pi$-base. Applying Lemma~\ref{7.2}(2), we get that $D\setminus AC(f)$ is dense in $X$.

By induction on the tree $\w^{<\w}$ we shall construct sequences $(x_s)_{s\in\w^{<\w}}$ of points of the set $D\setminus AC(f)$, and sequences $(V_s)_{s\in \w^{<\w}}$ and $(U_s)_{s\in\w^{<\w}}$, $(W_s)_{s\in\w^{<\w}}$ of sets so that the following conditions hold for every finite number sequence $s\in\w^{<\w}$:
\begin{itemize}
\item[(a)] $V_s$ is an open neighborhood of the point $x_s$ in $X$;
\item[(b)] $W_s\subset U_s$ are open neighborhoods of $f(x_s)$ in $Y$;
\item[(c)] $f(V_s)\subset U_s$;
\item[(d)] $V_{s\concat n}\subset V_s$ and $U_{s\concat n}\subset U_s$ for all $n\in\w$;
\item[(e)] the sequence $(V_{s\concat n})_{n\in\IN}$ converges to $x_s$;
\item[(f)] $W_{s}\cap U_{s\concat n}=\emptyset=U_{s\concat n}\cap U_{s\concat m}$ for all $n\ne m$ in $\w$.
\end{itemize}

We start the induction letting $V_\emptyset=X$, $U_\emptyset=Y$ and $x_\emptyset$ be any point of $D\setminus AC(f)$.

Assume that for a finite sequence $s\in\w^{<\w}$ the point $x_s\in D\setminus AC(f)$
and open sets $V_s\subset X$ and $U_s\subset Y$ with $x_s\in V_s$ and $f(V_s)\subset U_s$ have been constructed.  Since $f|V_s$ fails to be almost continuous at $x_s$, there is a neighborhood $W_s\subset U_s$ of $f(x_s)$ such that the closure of the preimage $f^{-1}(\overline{W_s})$ is not a neighborhood of $x_s$ in $X$. This fact and the Preiss-Simon property of $X$ at $x_s$ allows us to construct a sequence $(V'_{k})_{k\in\w}$ of open subsets of $V_s\setminus \cl_X\big(f^{-1}(\overline{W_s})\big)$ that converges to $x_s$ in the sense that each neighborhood of $x$ contains all but finitely many sets $V'_{k}$. Applying Lemma~\ref{7.2a}(2) to the map $f|V_s:V_s\to U_s$,
we can find an infinite subset $N\subset \w$ and a sequence $(U'_{k})_{k\in N}$ of pairwise disjoint open sets of $U_s$ such that each set $f^{-1}(U'_k)\cap V'_k$, $k\in N$, has non-empty interior in $X$. Let $N=\{k_n:n\in\w\}$ be the increasing enumeration of the set $N$.

For every $n\in\w$ let $U_{s \concat n}=U'_{k_n}\setminus\overline{W_s}$, $V_{s\concat n}$ be a non-empty open subset in $f^{-1}(U_{k_n})\cap V_{k_n}'$ and $x_{s\concat n}\in V_{s\concat n}\cap D\setminus AC(f)$ be any point (such a point exists because of the density of $D\setminus AC(f)$ in $X$). One can check that the points $x_{s\concat n}$, ${n\in\w}$ and sets $W_s$, $V_{s\concat n}$, $U_{s\concat n}$, $n\in\w$ satisfy the requirements of the inductive construction.

After completing the inductive construction, consider the set $Q=\{x_s:s\in\w^{<\w}\}$ and note that it is first countable, contains no non-empty finite open subsets, $f|Q$ is bijective and $f(Q)$ is a discrete subspace of $Y$.
\end{proof}

\section{Proof of Theorem~\ref{main}}\label{s:pf}

Given a function $f:X\to Y$ from a Preiss-Simon space $X$ to a regular space $Y$ we need to prove the equivalence of the following conditions:
\begin{enumerate}
\item $f$ is weakly discontinuous;
\item $f$ is resolvable;
\item $f$ is open-resolvable.
\item $f$ is base-resolvable.
\end{enumerate}

The implication $(1)\Ra(2)$ follows from Proposition~\ref{p1} and $(2)\Ra(3)\Ra(4)$ are trivial. To prove that $(4)\Ra(1)$, assume that the function $f$ is base-resolvable but not weakly discontinuous. By Lemma~\ref{l:scat}, $f$ is not scatteredly continuous. The space $X$, being Preiss-Simon, has countable tightness. Then Lemma~\ref{l:count} yields a non-empty countable set $D\subset X$ such that the restriction $f|D$ has no continuity points.   Applying
Lemma~\ref{7.2a}(3) to the restriction $f|D$, we can find a countable first-countable subset $Q\subset D$ without finite open sets such that $f|Q$ is bijective and $f(Q)$ is a discrete subspace of $Y$. It is clear that $f|Q$ has no continuity point. The space $X$ being second countable and without finite open sets, can be written as the union $Q=Q_1\cup Q_2$ of two disjoint dense subsets of $Q$.

Let $\mathcal B$ be a countably additive base for $Y$ such that for every $B\in\mathcal B$ the preimage $f^{-1}(B)$ is resolvable in $X$. Since the set $f(Q)$ is countable and discrete, for every $x\in Q$ we can select a neighborhood $O_{f(x)}\in\mathcal B$ of $f(x)$ so small that the family $\{O_{f(x)}:x\in Q\}$ is disjoint. The countable additivity of the base $\mathcal B$ implies that for every $i\in\{1,2\}$ the set $W_i=\bigcup_{x\in Q_i}O_{f(x)}$ belongs to $\mathcal B$. Consequently, the preimage $f^{-1}(W_i)$ is a resolvable subset of $X$ and hence $\bar Q\cap f^{-1}(W_i)\supset Q_i$ is a dense resolvable subset of $\bar Q$. So, the space $\bar Q$ contains two disjoint dense resolvable subsets $\bar Q\cap f^{-1}(W_1)$ and  $\bar Q\cap f^{-1}(W_2)$, which contradicts Lemma~\ref{resolve}. This contradiction completes the proof of the implication $(4)\Ra(1)$.
%\newpage

\end{document}